\documentclass[preprint,12pt,3p]{elsarticle}
\usepackage{amsmath,amsthm,amsfonts,amssymb}
\usepackage{algorithm}
\usepackage{fullpage}
\usepackage{subfig}
\usepackage{blkarray}
\usepackage{caption}
\usepackage{float}
\usepackage{algpseudocode}
\usepackage{tikz}
\usepackage{hyperref}
\numberwithin{equation}{section}
\usepackage{xcolor}
\usepackage{colortbl}
\usepackage[normalem]{ulem}
\usepackage{sectsty}
\definecolor{astral}{RGB}{46,116,181}
\subsectionfont{\color{astral}}
\sectionfont{\color{astral}}
\usepackage{colortbl}

\DeclareMathAlphabet{\mathpzc}{OT1}{pzc}{m}{it}
\DeclareFontFamily{OT1}{pzc}{}
\DeclareFontShape{OT1}{pzc}{m}{it}{<-> s * [0.900] pzcmi7t}{}
\DeclareMathAlphabet{\mathpzc}{OT1}{pzc}{m}{it}

\usepackage{amsbsy}
\usepackage{amsmath}
\usepackage{accents}
\newlength{\dhatheight}

\newenvironment{frcseries}{\fontfamily{frc}\selectfont}{}

\DeclareMathAlphabet\mathbfcal{OMS}{cmsy}{b}{n}
\definecolor{darkslategray}{rgb}{0.18, 0.31, 0.31}
\definecolor{warmblack}{rgb}{0.0, 0.26, 0.26}

\usepackage{multirow}
\usepackage{mathtools}
\usepackage{algorithm}
\usepackage{algorithmicx}
\makeatletter
\def\BState{\State\hskip-\ALG@thistlm}
\makeatother
\newtheorem{theorem}{Theorem}[section]
\newtheorem{lemma}[theorem]{Lemma}
\newtheorem{corollary}[theorem]{Corollary}
\theoremstyle{definition}
\usepackage{hyperref}
\newtheorem{definition}{Definition}[section]
\newtheorem{remark}{Remark}[section]
\newtheorem{example}{Example}[section]
\journal{XYZ}
\newcommand{\R}{{\mathbb R}}

\newcommand{\kronecker}{\raisebox{1pt}{\ensuremath{\:\otimes\:}}}

\begin{document}

\begin{frontmatter}

\title{Additional results on convergence and semiconvergence of three-step alternating iteration scheme for singular linear systems
}

\vspace{-.4cm}

\author{Vaibhav Shekhar$^a$, Punit Sharma$^b$}

\address{Department of Mathematics,\\
                        Indian Institute of Technology Delhi, India.
                        \\email$^a$: vaibhavshekhar29@gmail.com \\
                        email$^b$: punit.sharma@maths.iitd.ac.in \\
                    }
                        
\vspace{-2cm}

\begin{abstract}
 The three-step alternating iteration scheme for finding an iterative solution of a singular (non-singular) linear systems in a faster way was introduced by Nandi {\it et al.} [Numer. Algorithms;  84 (2) (2020) 457-483], recently.  The authors then provided its convergence criteria for a class of matrix splitting called proper G-weak regular splittings of type I. 
In this note, we analyze further the convergence criteria of the same scheme. In this aspect, we obtain sufficient conditions for the convergence of the same scheme for another class of matrix splittings called  proper G-weak regular splittings of type II. We then show that this scheme converges faster than the two-step alternating and usual iteration schemes, even for this class of splittings. As a particular case, we also establish faster convergence criteria of three-step in a nonsingular matrix setting.
This is shown that a large amount of computational time and memory are required in single-step and two-step alternating iterative methods to solve the nonsingular linear systems more efficiently than the three-step alternating iteration method. Finally, the semiconvergence of a three-step alternating iterative scheme is established. Its faster semiconvergence is demonstrated by considering a singular linear system arising from the Markov process.

\end{abstract}
\end{frontmatter}

\section{Introduction}
Singular linear systems of the form
\begin{equation}\label{e1}
    Ax=b,
\end{equation}
 where $A\in \mathbb{R}^{n\times n}$, $x\in \mathbb{R}^{n}$, and $b\in \mathbb{R}^{n}$, appear in many mathematical problems like the study of Markov processes (see
\cite{mey}) and finite difference methods for solving certain partial differential equations, such as the Neumann problem and Poisson's equation on a sphere (see  \cite{ple}). We first consider the case when $A$ has {\it index 1}, i.e., if rank($A$)=rank($A^2$). The  unique matrix $X\in {\R}^{n\times n}$ satisfying the matrix equations $AXA=A$, $XAX=X$ and $AX=XA$ is called the {\it group inverse}  of a matrix $A\in {\R}^{n\times n}$, and is denoted by $A^{\#}$. It does not always exist, but only for index one matrices. Using a matrix splitting, one may utilize this generalized inverse to obtain a solution of \eqref{e1}. An expression $A=U-V$ is called a splitting of $A$ if $U$ is nonsingular. The authors of \cite{bpcones} proposed a sub-class of splitting called proper splitting to solve a rectangular/ singular linear system. A splitting $A=U-V$ is called  {\it a proper splitting} \cite{bpcones} if $\mathcal{R}(U)=\mathcal{R}(A)$ and $\mathcal{N}(U) = \mathcal{N}(A)$, where $\mathcal{R}(B)$ and $\mathcal{N}(B)$ denote the range space and the null space of a matrix $B$, respectively. Mishra
and Mishra \cite{nmm} showed the uniqueness of a proper splitting under suitable sufficient
conditions. Using a proper splitting $A=U-V$, Wei \cite{wei:1998}  showed that the (usual) iteration scheme 
     \begin{equation}\label{eqn1.2}
       x^{k+1} = U^{\#}Vx^{k} + U^{\#}b, ~~~k=0,1,2,\cdots  
     \end{equation}
     converges to $A^{\#}b$ (called the minimum $P$-norm solution) if and only if $\rho(U^{\#}V)<1$ for solving an index 1  linear system of the form \eqref{e1}.  The author also established several convergence criteria of the iterative scheme \eqref{eqn1.2} for different sub-classes of proper splittings.  A proper splitting $A=U-V$ of $A\in \mathbb{R}^{n\times n}$ is called 
     \begin{enumerate}
         \item {\it proper G-regular splitting} if $U^{\#}$ exists, $U^{\#}\geq 0$ and $V\geq 0$\cite{baliarsingh:2015},
         \item {\it proper G-weak regular splitting (of type I)}  if  $U^{\#}$ exists, $U^{\#}\geq 0$ and $U^{\#}V\geq 0$\cite{baliarsingh:2015},
         \item {\it proper G-weak regular splitting of type II}  if $U^{\#}$ exists, $U^{\#}\geq 0$ and $VU^{\#}\geq 0$\cite{giri:2019}.
     \end{enumerate}
        Here $B\geq 0$ means all the entries of $B$ are nonnegative. An $n \times n$ real matrix $A$ is called {\it monotone}
 if $Ax \geq  0 ~\Rightarrow x \geq 0$. Collatz \cite{collatz:1966} showed that $A$ is monotone if and only if  $A^{-1}$ exists and  $A^{-1}\geq 0$. We refer the reader to \cite{collatz:1966}, which discusses the natural occurrence of monotone matrices in  finite-difference approximation methods for solving elliptic partial differential equations, and to \cite{bpbook}, which discusses different generalizations of the notion of monotonicity and their characterization. One such characterization is recalled next.  A real $n \times n$ matrix $A$ is called {\it group monotone}
 if $A^{\#}$ exists and $A^{\#}\geq 0$.  A characterization of a group monotone matrix in terms of proper G-weak regular splitting of type I is stated next. Let $A = U-V$ be a proper G-weak regular splitting of type I. Then $A$ is group monotone if and only if $\rho(U^{\#}V)<1$ \cite{nandi:2019}. This characterization also provides sufficient conditions for the convergence of the iteration scheme \eqref{eqn1.2}. Analog characterization of a group monotone matrix also exists in the literature regarding proper G-weak regular splitting of type II (see Section \ref{sec2} for details).
   The usual iteration scheme \eqref{eqn1.2} converges very slowly in many practical cases. So, different comparison results are established in the literature (see \cite{giri:2016}, \cite{jena:2018}, \cite{jena:2012}, \cite{jena:2013} and \cite{wei:2001}). These results help us to find a splitting whose corresponding iteration scheme converges faster than the other.  But, in the case of a matrix having many proper splittings, finding the best splitting (in the sense of faster convergence) is a complicated and time-consuming process. To avoid this, Nandi {\it et al.} \cite{nandi:2019} proposed the three-step alternating iteration scheme:
\begin{equation}\label{eqn1.6}
x^{k+1}=X^{\#}YU^{\#}VK^{\#}Lx^{k}+X^{\#}(YU^{\#}VK^{\#}+YU^{\#}+I)b,\quad k=0,1,2,\ldots,
\end{equation}
using proper splittings $A=K-L=U-V=X-Y$ motivated by the work of Mishra \cite{mishra:2018}. Similarly, one can generate the two-step alternating scheme 
\begin{equation}\label{eqn1.7}
x^{k+1}=X^{\#}YU^{\#}Vx^{k}+X^{\#}(YU^{\#}+I)b,\quad k=0,1,2,\ldots,
\end{equation}
corresponding to proper splittings $A=U-V=X-Y$, the group inverse version of the alternating iteration scheme (7) of \cite{mishra:2018}. The authors of \cite{nandi:2019} then showed that the three-step alternating scheme converges for proper G-weak regular splittings of type I under the assumption of certain conditions. However,  convergence theory for the iteration schemes \eqref{eqn1.6} and \eqref{eqn1.7} is not yet studied in the case of proper G-weak regular splittings of type II.  \\

Let us turn our attention to the case when the coefficient matrix in \eqref{e1} is not necessarily of index 1. It is well known that when the matrix $U$ is  nonsingular, then \eqref{eqn1.2} becomes
 \begin{equation}\label{e6.22}
       x^{k+1} = U^{-1}Vx^{k} + U^{-1}b.  
     \end{equation}
A necessary
and sufficient condition for the convergence of \eqref{eqn1.2} for any $x^0$ is that $\rho(U^{-1}V)< 1$ ($U^{-1}V$ is zero-convergent in this case) \cite{varga}. However, if $A$ is singular, then in this case, $1$ is an eigenvalue of $U^{-1}V$ and so the convergence of \eqref{e6.22} fails, i.e., $U^{-1}V$ is not zero-convergent, then
we look for some necessary and sufficient conditions so that \eqref{e6.22} converges to a solution of \eqref{e1}. If this happens, we say that the scheme \eqref{e6.22} is semiconvergent. 
Semiconvergence criteria for a three-step alternating iterative scheme are also not established in the literature.

The purpose of this article is two-fold. First, we establish the convergence theory of the three-step alternating method for proper G-weak regular splittings of type II. Second, we obtain sufficient criteria for semiconvergence of the three-step scheme in the case of quasi-weak regular splittings of type I or type II.
The rest of the article is organized in the following way. In Section \ref{sec2}, we introduce notations, definitions and a few preliminary results frequently used to derive the main results. In Section \ref{sec3}, we consider the problem of solving the index one linear system that comprises some results presenting another set of convergence criteria of the three-step alternating iteration scheme \eqref{eqn1.6}. This section also discusses comparison results which show that the three-step alternating iteration scheme converges faster than the two-step alternating iteration scheme and the usual iteration scheme. Section \ref{sec4} deals with the problem of semiconvergence of three-step alternating schemes when the coefficient matrix is a singular matrix or singular $M$-matrix.

\section{Preliminaries}\label{sec2}
Let $\mathbb{R}^{n\times n}$ denote the set of all real matrices of order $n\times n$. $\mathbb{R}^{n}$ denotes an $n$-dimensional Euclidean space. $A^t$ denotes the transpose of the matrix $A$. Let $L$ and $M$ be complementary subspaces of $\mathbb{R}^{n}$ (i.e., $L\oplus M=\mathbb{R}^{n}$), then $P_{L,M}$ is a projector on $L$ along $M$. So, $P_{L,M}A=A$ if and only if $\mathcal{R}(A)\subseteq L$ and $AP_{L,M}=A$ if and only if $\mathcal{N}(A)\supseteq M$. $r(A)$ denotes the rank of matrix $A$. Some useful properties of the group inverse are: $\mathcal{R}(A) = \mathcal{R}(A^{\#})$; $\mathcal{N}(A) = \mathcal{N}(A^{\#})$; $AA^{\#} = P_{\mathcal{R}(A),\mathcal{N}(A)}=A^{\#}A$. In particular, if $x\in \mathcal{R}(A)$, then $x = A^{\#}Ax$.
Let $A, B\in \mathbb{R}^{n\times n}$. Then $A\geq B$ means $A-B\geq 0$. A similar notation follows for $x\in \mathbb{R}^{n}$ as it can be treated as an $n\times 1$ matrix. The spectral radius of $A$ is denoted by $\rho(A)=\max\{|\lambda|: \lambda \text{ is an eigenvalue of } A\}$ and the spectrum is denoted by $\sigma(A)$. We define $\gamma(A)=\max\{|\lambda|:\lambda \in \sigma(A), \lambda\neq 1\}$. 
A matrix $A$ is called convergent (or zero-convergent) if $\displaystyle\lim_{k\to \infty}A^k=0$ and semiconvergent if $\displaystyle\lim_{k\to \infty}A^k$ exists, although it need not be the zero matrix. $A$ is called an $M$-matrix if $A=sI-B$ such that $s>0$ and $s\geq \rho(B)$. It is a nonsingular $M$-matrix if $s> \rho(B)$ and singular $M$-matrix if $s=\rho(B)$. An $M$-matrix $A$ is said to have property $c$ if $s^{-1}B$ is semiconvergent for some $s$. 
The following three results below deal with a nonnegative matrix and its spectral radius.

\begin{theorem}\label{frob1}\textnormal{(\cite[Theorem 2.20]{varga})}\\
Let $A \in {\R}^{n \times n}$ and $A \geq 0$. Then\\
$(i)$ A has a  real eigenvalue equal to its spectral radius.\\
$(ii)$ there exists a nonnegative eigenvector for its spectral radius.
 \end{theorem}
 
 \begin{theorem}\label{frob2}\textnormal{(\cite[Theorem 2.1.11]{bpbook})}\\
 Let $A \in {\R}^{n \times n}$, $A \geq 0$, $x\geq 0$ $(x\neq 0)$ and $\alpha$ be a positive scalar.\\
 $(i)$ If $\alpha x\leq Ax$, then $\alpha \leq \rho(A)$.\\
 $(ii)$ If $Ax\leq \alpha x$, $x>0$, then $\rho(A)\leq \alpha$.
 \end{theorem}

\begin{theorem}\label{neuman}\textnormal{(\cite[Theorem 3.15]{varga})}\\
Let $A\in {\R}^{n\times n}$ and $A\geq 0$. Then $\rho(A)<1$ if and only if  $(I-A)^{-1}$ exists and $(I-A)^{-1}=\displaystyle \sum_{k=0}^{\infty}A^{k}\geq 0$.
\end{theorem}

\cite[Theorem 4]{nandi:2019} collects a few properties of a proper splitting $A=U-V$, and is presented next. 

\begin{theorem}\textnormal{(\cite[Theorem 4]{nandi:2019})}\label{thm2.4}\\
Let $A=U-V$ be a proper splitting of $A\in \mathbb{R}^{n\times n}$. Suppose that $A^{\#}$ exists. Then\\
$(i)$ $U^{\#}$ exists,\\
$(ii)$ $AA^{\#}=UU^{\#}$; $A^{\#}A=U^{\#}U$,\\
$(iii)$ $A=U(I-U^{\#}V)=(I-VU^{\#})U$,\\
$(iv)$ $I-U^{\#}V$ and $I-VU^{\#}$ are non-singular,\\
$(v)$ $A^{\#}=(I-U^{\#}V)^{-1}U^{\#}=U^{\#}(I-VU^{\#})^{-1}$.
\end{theorem}

The following result is a particular case of \cite[Theorem 10]{giri:2019} and provides convergence criteria for the iteration scheme \eqref{eqn1.2}.

 \begin{theorem}\label{thm3.7}
 Let $A=U-V$ be a proper G-weak regular splitting of type II of $A$. Then $A^{\#}\geq 0$ if and only if $\rho(U^{\#}V)<1$.
 \end{theorem}
 
 It is well-known that a matrix splitting yielding a smaller spectral radius of the iteration matrix has a faster convergence rate. The next result compares the convergence rate of different iteration matrices arising from different types of matrix splittings and is a particular case of \cite[Theorem 11]{giri:2019}.
 
 \begin{theorem}\label{thm3.8} Let $A=U-V$ be a proper G-weak regular splitting of type II and $A=X-Y$ be a proper G-weak regular splitting of type I of a group monotone matrix $A$. If $U^{\#}\geq X^{\#}$, then $\rho(U^{\#}V)\leq  \rho(X^{\#}Y)<1$.
 \end{theorem}

\section{Convergence of three-step alternating scheme for $G$-weak regular splittings of type II}\label{sec3}

This section provides convergence criteria of a three-step alternating scheme for proper $G$-weak regular splitting of type II. This section assumes that $A$ has index one unless stated explicitly. We frequently use $\mathcal{H}$ for the iteration matrix of the corresponding iteration scheme unless stated. For example, if we consider proper G-weak regular splittings of type II $A=K-L=U-V=X-Y$, then $\mathcal{H}=X^{\#}YU^{\#}VK^{\#}L$, and $\mathcal{S}=YX^{\#}VU^{\#}LK^{\#}$ is the nonnegative matrix. Likewise, in case of proper G-weak regular splittings of type II $A=U-V=X-Y$, $\mathcal{H}=X^{\#}YU^{\#}V$, and $\mathcal{S}=YX^{\#}VU^{\#}$. The following two results help derive the convergence criteria of the three-step iterative scheme \eqref{eqn1.6} when the given matrix splittings are proper G-weak regular splittings of type II.

\begin{lemma}\label{lem4.2}
Let $A=K-L=U-V=X-Y$ be three proper splittings of $A$. If $B=K(K+X-A+YU^{\#}L)^{\#}X$, $\mathcal{R}(K+X-A+YU^{\#}L)=\mathcal{R}(A)$ and $\mathcal{N}(K+X-A+YU^{\#}L)=\mathcal{N}(A)$, then $B^{\#}=X^{\#}(K+X-A+YU^{\#}L)K^{\#}$, $\mathcal{R}(B)=\mathcal{R}(A)$ and $\mathcal{N}(B)=\mathcal{N}(A).$
\end{lemma}

\begin{proof}
Let $Z=X^{\#}(K+X-A+YU^{\#}L)K^{\#}.$ Since $\mathcal{R}(K+X-A+YU^{\#}L)=\mathcal{R}(A)$, $\mathcal{N}(K+X-A+YU^{\#}L)=\mathcal{N}(A)$ and $A=K-L=U-V=X-Y$ are proper splittings, we have $(K+X-A+YU^{\#}L)(K+X-A+YU^{\#}L)^{\#}=AA^{\#}=KK^{\#}=UU^{\#}=XX^{\#}$. Then,
\begin{align*}
ZB&=X^{\#}(K+X-A+YU^{\#}L)K^{\#}K(K+X-A+YU^{\#}L)^{\#}X\\
&=X^{\#}(K+X-A+YU^{\#}L)(K+X-A+YU^{\#}L)^{\#}X\\
&=X^{\#}X=XX^{\#}=KK^{\#}\\
&=K(K+X-A+YU^{\#}L)^{\#}(K+X-A+YU^{\#}L)K^{\#}\\
&=K(K+X-A+YU^{\#}L)^{\#}XX^{\#}(K+X-A+YU^{\#}L)K^{\#}\\
&=BZ.
\end{align*}
Now, $ZB=X^{\#}X$ and $Z=X^{\#}(K+X-A+YU^{\#}L)K^{\#}$ yield $ZBZ=X^{\#}XX^{\#}(K+X-A+YU^{\#}L)K^{\#}
=X^{\#}(K+X-A+YU^{\#}L)K^{\#}=Z$. Similarly, using $BZ=KK^{\#}$ and $B=K(K+X-A+YU^{\#}L)^{\#}X$, we have $BZB=KK^{\#}K(K+X-A+YU^{\#}L)^{\#}X=K(K+X-A+YU^{\#}L)^{\#}X=B$. Therefore, $Z=B^{\#}.$ \par
In order to show that $\mathcal{R}(B)=\mathcal{R}(A)$ and $\mathcal{N}(B)=\mathcal{N}(A)$,  we first prove that $\mathcal{N}(X)=\mathcal{N}(B)$ as $\mathcal{N}(A)=\mathcal{N}(X)$ follows from the fact that $A=X-Y$ is a proper splitting. Clearly,  $\mathcal{N}(X)\subseteq \mathcal{N}(B)$. Let $Bx=0$.
Pre-multiplying by $K^{\#}$ and using $K^{\#}K=(K+X-A+YU^{\#}L)^{\#}(K+X-A+YU^{\#}L)$, we get $(K+X-A+YU^{\#}L)^{\#}Xx=0$. Again, pre-multiplying by $(K+X-A+YU^{\#}L)$ and using $(K+X-A+YU^{\#}L)(K+X-A+YU^{\#}L)^{\#}=XX^{\#}$ yields $Xx=0$. So, $\mathcal{N}(X)=\mathcal{N}(B)$. Next, we have to show that $\mathcal{R}(A)=\mathcal{R}(B)$ which is  equivalent to $\mathcal{N}(A^{T})=\mathcal{N}(B^{T})$. Again, the fact $A=K-L$ is a proper splitting yields $\mathcal{N}(A^{T})=\mathcal{N}(K^{T})$. So, it is required to prove that  $\mathcal{N}(K^{T})=\mathcal{N}(B^{T})$. Clearly, $\mathcal{N}(K^{T})\subseteq \mathcal{N}(B^{T})$. Let $B^{T}x=0$. Pre-multiplying by $(X^{\#})^{T}$ we get $(X^{\#})^{T}X^{T}((K+X-A+YU^{\#}L)^{\#})^{T}K^{T}=0$, i.e., $(K(K+X-A+YU^{\#}L)^{\#}XX^{\#})^{T}x=0$. Since $XX^{\#}=(K+X-A+YU^{\#}L)(K+X-A+YU^{\#}L)^{\#}$, we have $((K+X-A+YU^{\#}L)^{\#})^{T}K^{T}x=0$. Again, pre-multiplying by $(K+X-A+YU^{\#}L)^{T}$ and using $K^{\#}K=(K+X-A+YU^{\#}L)^{\#}(K+X-A+YU^{\#}L)$, we get $K^{T}x=0$. Thus $\mathcal{N}(B^{T})=\mathcal{N}(K^{T})=\mathcal{N}(A^{T})$. Hence $\mathcal{R}(A)=\mathcal{R}(B)$.
 \end{proof}

In Lemma \ref{lem4.2}, $B^{\#}$ can also be expressed as
\begin{align}\label{eqn4.1}
B^{\#}&=X^{\#}(K+X-A+YU^{\#}L)K^{\#}\notag\\
&=X^{\#}KK^{\#}+X^{\#}XK^{\#}-X^{\#}AK^{\#}+X^{\#}YU^{\#}LK^{\#}\notag\\
&=X^{\#}+X^{\#}YK^{\#}+X^{\#}YU^{\#}LK^{\#}
\end{align}
and 
\begin{align}\label{eqn4.2}
  B^{\#}&=K^{\#}+X^{\#}LK^{\#}+X^{\#}YU^{\#}LK^{\#}. 
\end{align}

In the case of the two-step alternating iteration, we have the following result: a group version of \cite[Lemma 4.3]{giri:2017}.

\begin{corollary}
Let $A=U-V=X-Y$ be two proper splittings of $A$. If $B=U(U+X-A)^{\#}X$, $\mathcal{R}(U+X-A)=\mathcal{R}(A)$ and $\mathcal{N}(U+X-A)=\mathcal{N}(A)$, then $B^{\#}=X^{\#}(U+X-A)U^{\#}$, $\mathcal{R}(B)=\mathcal{R}(A)$ and $\mathcal{N}(B)=\mathcal{N}(A).$
\end{corollary}

When we have proper G-weak regular splittings of type II, establishing convergence criteria for the iterative scheme \eqref{eqn1.6} does not seem straightforward. Therefore, we seek the help of the  matrix $\mathcal{S}=YX^{\#}VU^{\#}LK^{\#}$. The next result is about some important relations that $\mathcal{H}=X^{\#}YU^{\#}VK^{\#}L$ and $\mathcal{S}$ possess.

\begin{theorem}\label{thm4.3}
Let $A=K-L=U-V=X-Y$ be three proper splittings of $A$ and $\mathcal{S}=YX^{\#}VU^{\#}LK^{\#}.$ Then,\\
$(i)$ $AA^{\#}\mathcal{S}=\mathcal{S}=\mathcal{S}AA^{\#}$ and $A^{\#}A\mathcal{H}=\mathcal{H}=\mathcal{H}A^{\#}A$, where $\mathcal{H}$ is the iteration matrix of the iteration scheme \eqref{eqn1.6}.\\
$(ii)$ $\mathcal{S}=A\mathcal{H}A^{\#},$ $\mathcal{H}=A^{\#}\mathcal{S}A$ and $\rho(\mathcal{S})=\rho(\mathcal{H}).$\\
$(iii)$ $I-\mathcal{S}$ and $I-\mathcal{H}$ are invertible if $\mathcal{R}(K+X-A+YU^{\#}L)=\mathcal{R}(A)$ and $\mathcal{N}(K+X-A+YU^{\#}L)=\mathcal{N}(A).$
\end{theorem}

\begin{proof}
$(i)$ Clearly, we have $\mathcal{R}(\mathcal{S})\subseteq \mathcal{R}(Y)$. But $\mathcal{R}(Y)\subseteq \mathcal{R}(A)$ as  $A=X-Y$ is a proper splitting.  So, $\mathcal{R}(\mathcal{S})\subseteq \mathcal{R}(Y)\subseteq \mathcal{R}(A)$. Thus, $AA^{\#}\mathcal{S}=\mathcal{S}$. Similarly, since $A=X-Y$ is a proper splitting, it is easy to see that $\mathcal{S}AA^{\#}=YX^{\#}VU^{\#}LK^{\#}AA^{\#}=YX^{\#}VU^{\#}LK^{\#}KK^{\#}=\mathcal{S}$. Applying a similar argument, we get the other equality.\\
$(ii)$ Since $AA^{\#}=XX^{\#}$, we get $X^{\#}YA^{\#}=X^{\#}(X-A)A^{\#}=X^{\#}XA^{\#}-X^{\#}AA^{\#}=A^{\#}AA^{\#}-X^{\#}XX^{\#}=A^{\#}XX^{\#}-A^{\#}AX^{\#}=A^{\#}(X-A)X^{\#}=A^{\#}YX^{\#}$. Similarly, $U^{\#}VA^{\#}=A^{\#}VU^{\#}$ and $K^{\#}LA^{\#}=A^{\#}LK^{\#}$. So,
\begin{align*}
\mathcal{S}&=AA^{\#}\mathcal{S}\\
&=AA^{\#}YX^{\#}VU^{\#}LK^{\#}\\
&=AX^{\#}YA^{\#}VU^{\#}LK^{\#}\\
&=AX^{\#}YU^{\#}VA^{\#}LK^{\#}\\
&=AX^{\#}YU^{\#}VK^{\#}LA^{\#}\\
&=A\mathcal{H}A^{\#}.
\end{align*}
\allowdisplaybreaks
We then have $A^{\#}\mathcal{S}=A^{\#}A\mathcal{H}A^{\#}=\mathcal{H}A^{\#}$.  Now, post-multiplying by $A$, we get $\mathcal{H}=A^{\#}\mathcal{S}A$. To show $\rho(\mathcal{S})=\rho(\mathcal{H})$, consider the eigenvalue problem $\mathcal{S}x=\lambda x$, where $\lambda$ is any non-zero eigenvalue of $\mathcal{S}$, and $x$ is the corresponding eigenvector. Then, $x\in \mathcal{R}(\mathcal{S})\subseteq \mathcal{R}(A)$. Now, $\lambda x=\mathcal{S}x=A\mathcal{H}A^{\#}x$. Pre-multiplying by $A^{\#},$ we get $\lambda y=\mathcal{H}y$, where $y=A^{\#}x.$ Therefore, $\lambda$ is an eigenvalue of $\mathcal{H}$ if $y\neq0$. Suppose that $y=A^{\#}x=0$. Then, $x \in  \mathcal{N}(A^{\#})=\mathcal{N}(A).$ So, we have $x\in \mathcal{R}(A)\cap \mathcal{N}(A)$. Since $index(A)=1$, we have $\mathcal{R}(A)\cap \mathcal{N}(A)=\{0\}$. This implies that $x=0$ is a contradiction. Hence $y\neq 0$. Thus, $\sigma(\mathcal{S})\subseteq \sigma(\mathcal{H})$. For the other way, consider $\mathcal{H}y=\mu y$, where $\mu$ is any non-zero eigenvalue of $\mathcal{H}$, and $y$ is its corresponding eigenvector. Then, $y\in \mathcal{R}(\mathcal{H})\subseteq \mathcal{R}(A)$. Since $\mathcal{H}=A^{\#}\mathcal{S}A$, we have, $\mu y=A^{\#}\mathcal{S}Ay$. Now, pre-multiplying by A, we get $\mu z=\mathcal{S}z$, where $z=Ay$. If $z=0$, then $y\in \mathcal{N}(A)$ which implies $y\in \mathcal{R}(A)\cap \mathcal{N}(A)$ which further yields $y=0$, a contradiction. Thus, $\sigma(\mathcal{H})\subseteq \sigma(\mathcal{S})$. Hence $\rho(\mathcal{S})=\rho(\mathcal{H})$.\\
$(iii)$ We will prove this by the method of contradiction. Assume that $I-\mathcal{S}$ is singular. Then, 1 is an eigenvalue of $\mathcal{S}$. Therefore, there exists $x\neq 0$ such that $x=\mathcal{S}x$ which implies $x\in \mathcal{R}(\mathcal{S})\subseteq \mathcal{R}(Y) \subseteq \mathcal{R}(A)=\mathcal{R}(X)=\mathcal{R}(U)=\mathcal{R}(K)$ and thus $x=AA^{\#}x=XX^{\#}x=UU^{\#}x=KK^{\#}x$. Then,
\begin{eqnarray*}
x=\mathcal{S}x&=&YX^{\#}VU^{\#}LK^{\#}x\\
&=&(X-A)X^{\#}(U-A)U^{\#}(K-A)K^{\#}x \\
&=&(XX^{\#}-AX^{\#})(UU^{\#}-AU^{\#})(KK^{\#}-AK^{\#})x\\
&=&(AA^{\#}-AX^{\#})(AA^{\#}-AU^{\#})(AA^{\#}-AK^{\#})x\\
&=&(AA^{\#}-AX^{\#})(AA^{\#}-AU^{\#})(AA^{\#}x-AK^{\#}x)\\
&=&(AA^{\#}-AX^{\#})(AA^{\#}-AU^{\#})(x-AK^{\#}x)\\
&=&(AA^{\#}-AX^{\#})(AA^{\#}x-AA^{\#}AK^{\#}x-AU^{\#}x+AU^{\#}AK^{\#}x)\\
&=&(AA^{\#}-AX^{\#})(x-AK^{\#}x-AU^{\#}x+AU^{\#}AK^{\#}x)\\
&=&(x-AK^{\#}x-AU^{\#}x+AU^{\#}AK^{\#}x-AX^{\#}x+AX^{\#}AK^{\#}x+AX^{\#}AU^{\#}x\\
&&-AX^{\#}AU^{\#}AK^{\#}x)\\
&=&x-A(K^{\#}+U^{\#}-U^{\#}AK^{\#}+X^{\#}-X^{\#}AK^{\#}-X^{\#}AU^{\#}\\
&&+X^{\#}AU^{\#}AK^{\#})x\\
&=&x-A(K^{\#}KK^{\#}+U^{\#}UU^{\#}+X^{\#}-U^{\#}AK^{\#}-X^{\#}AK^{\#}-X^{\#}AU^{\#}\\
&&+X^{\#}AU^{\#}AK^{\#})x\\
&=&x-A(X^{\#}XK^{\#}+X^{\#}XU^{\#}+X^{\#}-U^{\#}AK^{\#}-X^{\#}AK^{\#}-X^{\#}AU^{\#}\\
&&+X^{\#}AU^{\#}AK^{\#})x\\
&=& x-A(X^{\#}XK^{\#}+X^{\#}XU^{\#}UU^{\#}+X^{\#}XX^{\#}-U^{\#}UU^{\#}AK^{\#}-X^{\#}AK^{\#}\\
&&-X^{\#}AU^{\#}UU^{\#}+X^{\#}AU^{\#}AK^{\#})x\\
&=&x-A(X^{\#}XK^{\#}+X^{\#}XU^{\#}KK^{\#}+X^{\#}KK^{\#}-X^{\#}XU^{\#}AK^{\#}-X^{\#}AK^{\#}\\
&&-X^{\#}AU^{\#}KK^{\#}+X^{\#}AU^{\#}AK^{\#})x\\
&=&x-A(X^{\#}(X+K-A+XU^{\#}K-XU^{\#}A-AU^{\#}K+AU^{\#}A)K^{\#})x\\
&=&x-A(X^{\#}(X+K-A+YU^{\#}L)K^{\#})x\\
&=&x-AB^{\#}x,
\end{eqnarray*}
where $B=K(X+K-A+YU^{\#}L)^{\#}X$. By Lemma \ref{lem4.2}, given $\mathcal{R}(K+X-A+YU^{\#}L)=\mathcal{R}(A)$ and $\mathcal{N}(K+X-A+YU^{\#}L)=\mathcal{N}(A)$, we have $\mathcal{R}(A)=\mathcal{R}(B)$ and $\mathcal{N}(A)=\mathcal{N}(B)$. Now,
$AB^{\#}x=0$ implies $B^{\#}x\in \mathcal{N}(A)=\mathcal{N}(B)$ and thus $BB^{\#}x=0.$  Again, $x\in \mathcal{R}(A)=\mathcal{R}(B)$ yields $x=BB^{\#}x=0$, a contradiction. Thus, $I-\mathcal{S}$ is invertible.
Now, to prove other part assume that $I-\mathcal{H}$ is singular, then there exists $x\neq 0$ such that $x=\mathcal{H}x.$ So, $x\in \mathcal{R}(X^{\#})=\mathcal{R}(X)=\mathcal{R}(A).$ Substituting $Y=X-A$, $V=U-A$ and $L=K-A$ in $x=\mathcal{H}x$ and then simplifying as in first part, we get $B^{\#}Ax=0.$ Pre-multiplying by $B$ and using the fact $\mathcal{R}(B)=\mathcal{R}(A),$ we get $Ax=0$ which implies $x \in \mathcal{N}(A)\cap \mathcal{R}(A).$ Thus, $x=0$, a contradiction.
Hence $I-\mathcal{H}$ is invertible.
\end{proof}

When we have two proper splittings, the above result reduces to the following form.

\begin{corollary}
Let $A=U-V=X-Y$ be two proper splittings of $A$ and $\mathcal{S}=YX^{\#}VU^{\#}.$ Then,\\
$(i)$ $AA^{\#}\mathcal{S}=\mathcal{S}=\mathcal{S}AA^{\#}$ and $A^{\#}A\mathcal{H}=\mathcal{H}=\mathcal{H}A^{\#}A$, where $\mathcal{H}$ is the iteration matrix of the iteration scheme \eqref{eqn1.7}.\\
$(ii)$ $\mathcal{S}=A\mathcal{H}A^{\#},$ $\mathcal{H}=A^{\#}\mathcal{S}A$ and $\rho(\mathcal{S})=\rho(\mathcal{H}).$\\
$(iii)$ $I-\mathcal{S}$ and $I-\mathcal{H}$ are invertible if $\mathcal{R}(U+X-A)=\mathcal{R}(A)$ and $\mathcal{N}(U+X-A)=\mathcal{N}(A).$
\end{corollary}

Now, we provide  the  convergence result for the iterative scheme \eqref{eqn1.6} when the given splittings are proper G-weak regular splittings of type II.

\begin{theorem}\label{thm4.4}Let $A=K-L=U-V=X-Y$ be three proper G-weak regular splittings of type II of a group monotone matrix $A$ and $\mathcal{S}=YX^{\#}VU^{\#}LK^{\#}$. Then $\rho(\mathcal{H})=\rho(X^{\#}YU^{\#}VK^{\#}L)<1$.
 \end{theorem}

\begin{proof}
Since $A=K-L=U-V=X-Y$ are proper G-weak regular splittings of type II, we have
\begin{eqnarray*}
0\leq \mathcal{S}&=&YX^{\#}VU^{\#}LK^{\#}\\
&=&(X-A)X^{\#}(U-A)U^{\#}(K-A)K^{\#}\\
&=&(XX^{\#}-AX^{\#})(UU^{\#}-AU^{\#})(KK^{\#}-AK^{\#})\\
&=&(AA^{\#}-AX^{\#})(AA^{\#}-AU^{\#})(AA^{\#}-AK^{\#})\\
&=&AA^{\#}-AK^{\#}-AU^{\#}-AX^{\#}+AU^{\#}AK^{\#}+AX^{\#}AK^{\#}+AX^{\#}AU^{\#}\\
&&-AX^{\#}AU^{\#}AK^{\#}.\end{eqnarray*}
Then
\begin{eqnarray*}
A^{\#}(I-\mathcal{S})&=&A^{\#}-A^{\#}\mathcal{S}\\
&=&A^{\#}AA^{\#}-A^{\#}\mathcal{S}\\
&=&A^{\#}(AA^{\#}-\mathcal{S})\\
&=&A^{\#}(AA^{\#}-AA^{\#}+AK^{\#}+AU^{\#}+AX^{\#}-AU^{\#}AK^{\#}-AX^{\#}AK^{\#}\\
&&-AX^{\#}AU^{\#}+AX^{\#}AU^{\#}AK^{\#})\\
&=&A^{\#}AK^{\#}+A^{\#}AU^{\#}+A^{\#}AX^{\#}-A^{\#}AU^{\#}AK^{\#}-A^{\#}AX^{\#}AK^{\#}\\
&&-A^{\#}AX^{\#}AU^{\#}+A^{\#}AX^{\#}AU^{\#}AK^{\#}\\
&=&K^{\#}+U^{\#}+X^{\#}-U^{\#}AK^{\#}-X^{\#}AK^{\#}-X^{\#}AU^{\#}+X^{\#}AU^{\#}AK^{\#}\\
&=&K^{\#}+U^{\#}(K-A)K^{\#}+X^{\#}(U-A)U^{\#}+X^{\#}(AU^{\#}A-A)K^{\#}\\
&=&K^{\#}+U^{\#}LK^{\#}+X^{\#}VU^{\#}+X^{\#}(AU^{\#}A-UU^{\#}A)K^{\#}\\
&=&K^{\#}+U^{\#}LK^{\#}+X^{\#}VU^{\#}-X^{\#}(U-A)U^{\#}AK^{\#}\\
&=&K^{\#}+U^{\#}LK^{\#}+X^{\#}VU^{\#}-X^{\#}VU^{\#}(K-L)K^{\#}\\
&=&K^{\#}+U^{\#}LK^{\#}+X^{\#}VU^{\#}LK^{\#}\geq 0.
\end{eqnarray*}
Therefore,
\begin{align*}
0&\leq A^{\#}(I-\mathcal{S})(I+\mathcal{S}+\mathcal{S}^{2}+\mathcal{S}^{3}+...+\mathcal{S}^{m})\\
&=A^{\#}(I-\mathcal{S}^{m+1})\\
&\leq A^{\#},
\end{align*}
for each  integer
$m$, and it follows
that as $m\to \infty$, the entries of the matrix
$I+\mathcal{S}+\mathcal{S}^2+\dots+\mathcal{S}^m$
remains bounded.
Since
$\mathcal{S}\geq 0$, therefore, 
the sequence of partial sums of the series
$\displaystyle \sum_{m=0}^{\infty}\mathcal{S}^{m}$
converges and consequently,   $\displaystyle \lim_{m\to \infty}\mathcal{S}^m=0$.
Thus $\rho(\mathcal{S})<1$ resulting $\rho(\mathcal{H})<1$ by Theorem \ref{thm4.3}.
 \end{proof}

In the non-singular matrix $A$, we have the following result.

\begin{corollary}
Let $A=K-L=U-V=X-Y$ be three weak regular splittings of type II of a monotone matrix $A$ and $\mathcal{S}=YX^{-1}VU^{-1}LK^{-1}$, then $\rho(\mathcal{H})=\rho(X^{-1}YU^{-1}VK^{-1}L)<1.$
\end{corollary} 

Similarly, one can obtain  the  convergence result for the iterative scheme \eqref{eqn1.7} when the given two splittings are proper G-weak regular splittings of type II.

\begin{corollary}\label{cor3.7}
Let $A=U-V=X-Y$ be two proper G-weak regular splittings of type II of a group monotone matrix $A$. Then $\rho(\mathcal{H})=\rho(X^{\#}YU^{\#}V)<1$.
 \end{corollary}

The following result is the non-singular matrix version of Corollary \ref{cor3.7}. 
\begin{corollary}\textnormal{(Theorem 3, \cite{climent:2003})}\\
Let $A=U-V=X-Y$ be two weak regular splittings of type II of a monotone matrix $A$, then $\rho(\mathcal{H})=\rho(X^{-1}YU^{-1}V)<1.$
\end{corollary} 

The following simple example shows that the converse of Theorem \ref{thm4.4} is not true.

\begin{example}\label{ex3.10}
Let $A=
\begin{bmatrix}
1 & 0 &1\\
-2 &4 &-2\\
0 &0 &0
\end{bmatrix}$. Then, $A^{\#}=
\begin{bmatrix}
1 &0 &1\\
0.5 &0.25 &0.5\\
0 &0 &0
\end{bmatrix}\geq 0$. Clearly,
\begin{align*}
A&=\begin{bmatrix}
0.5 &0 &0.5\\
-6 &12 &-6\\
0 &0 &0
\end{bmatrix}-
\begin{bmatrix}
-0.5 &0 &-0.5\\
-4 &8 &-4\\
0 &0 &0
\end{bmatrix}=K-L\\
&=\begin{bmatrix}
0.5 &0 &0.5\\
-8 &16 &-8\\
0 &0 &0
\end{bmatrix}-
\begin{bmatrix}
-0.5 &0 &-0.5\\
-6 &12 &-6\\
0 &0 &0
\end{bmatrix}=U-V\\
&=\begin{bmatrix}
0.8 &0 &0.8\\
-4 &8 &-4\\
0 &0 &0
\end{bmatrix}-
\begin{bmatrix}
-0.2 &0 &-0.2\\
-2 &4 &-2\\
0 &0 &0
\end{bmatrix}=X-Y
\end{align*}
 are proper splittings of $A$ and $\rho(\mathcal{H})=\rho(X^{\#}YU^{\#}VK^{\#}L)=0.25<1$.  But,\\
$U^{\#}=
\begin{bmatrix}
2 &0 &2\\
1 &0.0833 &1\\
0 &0 &0
\end{bmatrix}\geq 0$ and $VU^{\#}=
\begin{bmatrix}
-1 &0 &-1\\
0 &0.6667 &0\\
0 &0 &0
\end{bmatrix}\ngeq 0$,\\
\newline
$K^{\#}=
\begin{bmatrix}
2 &0 &2\\
1 &0.0625 &1\\
0 &0 &0
\end{bmatrix}\geq 0$ and $LK^{\#}=
\begin{bmatrix}
-1 &0 &-1\\
0 &0.75 &0\\
0 &0 &0
\end{bmatrix}\ngeq 0$,\\
\newline
$X^{\#}=
\begin{bmatrix}
1.25 &0 &1.25\\
0.625 &0.125 &0.625\\
0 &0 &0
\end{bmatrix}\geq 0$ and $YX^{\#}=
\begin{bmatrix}
-0.25 &0 &-0.25\\
0 &0.5 &0\\
0 &0 &0
\end{bmatrix}\ngeq 0$.
Hence, the individual splittings of $A$ are not proper G-weak regular splittings of type II.
\end{example}

Our next result tells us about the properties inherited by the iteration matrix of the iteration scheme \eqref{eqn1.6} when the given splittings are proper G-weak regular splittings of type II.

\begin{lemma}\label{lem4.8}
Let $A=K-L=U-V=X-Y$ be three proper G-weak regular splittings of type II of $A$ such that $\rho(\mathcal{H})<1$. If $\mathcal{R}(K+U-A+YU^{\#}L)=\mathcal{R}(A)$ and $\mathcal{N}(K+U-A+YU^{\#}L)=\mathcal{N}(A)$, then the unique splitting $A=B-C$ induced by $\mathcal{H}$ is a proper splitting such that $\mathcal{H}=B^{\#}C,$ where $B=A(I-\mathcal{H})^{-1}.$
\end{lemma}

\begin{proof}
Let $Z=(I-\mathcal{H})A^{\#}.$ Then \begin{align*}
    ZB&=(I-\mathcal{H})A^{\#}A(I-\mathcal{H})^{-1}\\
      &=(A^{\#}A-\mathcal{H}A^{\#}A)(I-\mathcal{H})^{-1}\\
      &=(A^{\#}A-A^{\#}A\mathcal{H})(I-\mathcal{H})^{-1}\\
      &=A^{\#}A=AA^{\#}\\
      &=A(I-\mathcal{H})^{-1}(I-\mathcal{H})A^{\#}\\
      &=BZ,
\end{align*}
which further yields $BZB=AA^{\#}A(I-\mathcal{H})^{-1}=A(I-\mathcal{H})^{-1}=B$ and $ZBZ=A^{\#}A(I-\mathcal{H})A^{\#}=(I-\mathcal{H})A^{\#}AA^{\#}=Z.$\\
 So, we have
\begin{eqnarray*}
B^{\#}&=&(I-\mathcal{H})A^{\#}\\
    &=&A^{\#}-X^{\#}YU^{\#}VK^{\#}LA^{\#}\\
    &=&A^{\#}-(X^{\#}(X-A)U^{\#}(U-A)K^{\#}(K-A)A^{\#})\\
    &=&A^{\#}-(U^{\#}-X^{\#}AU^{\#})(UK^{\#}-AK^{\#})(KA^{\#}-AA^{\#})\\ 
    &=&A^{\#}-(K^{\#}-U^{\#}AK^{\#}-X^{\#}AK^{\#}+X^{\#}AU^{\#}AK^{\#})(KA^{\#}-AA^{\#})\\
    &=&K^{\#}+U^{\#}-U^{\#}AK^{\#}+X^{\#}-X^{\#}AK^{\#}-X^{\#}AU^{\#}+X^{\#}AU^{\#}AK^{\#}\\
    &=&K^{\#}+X^{\#}-X^{\#}AK^{\#}+U^{\#}-U^{\#}(U-V)K^{\#}-X^{\#}(X-Y)U^{\#}\\
               &&+X^{\#}(X-Y)U^{\#}(K-L)K^{\#}\\
    &=&K^{\#}+X^{\#}-X^{\#}AK^{\#}+X^{\#}YU^{\#}LK^{\#}\\
    &=&X^{\#}(K+X-A+YU^{\#}L)K^{\#}.\end{eqnarray*}
As $\mathcal{R}(K+U-A+YU^{\#}L)=\mathcal{R}(A)$ and $\mathcal{N}(K+U-A+YU^{\#}L)=\mathcal{N}(A)$, we get that $A=B-C$ is a proper splitting by Lemma \ref{lem4.2}. Now,
\begin{align*}
    B^{\#}C&=B^{\#}(B-A)\\
                &=B^{\#}B-B^{\#}A\\
                &=B^{\#}B-(I-\mathcal{H})A^{\#}A\\
                &=B^{\#}B-A^{\#}A+\mathcal{H}A^{\#}A\\
                &=A^{\#}A-A^{\#}A+\mathcal{H}A^{\#}A\\
                &=\mathcal{H}.
\end{align*}
Next, we have to prove that $A=B-C$ is unique. Suppose that there exists another induced splitting $A=\overline{B}-\overline{C}$ such that $\mathcal{H}=\overline{B}^{\#}\overline{C}.$ Then $\overline{B}\mathcal{H}=$ $\overline{B}$ $\overline{B}^{\#}$ $\overline{C}=\overline{C}=\overline{B}-A.$ So, $\overline{B}(I-\mathcal{H})=A.$ Hence $\overline{B}=A(I-\mathcal{H})^{-1}=B.$ Therefore, $\mathcal{H}$ induces the unique proper splitting $A=B-C.$
 \end{proof}

In the case of two proper G-weak regular splittings of type II, we have the following result.

\begin{corollary}
Let $A=U-V=X-Y$ be proper G-weak regular splittings of type II of $A$ such that $\rho(\mathcal{H})<1$. If $\mathcal{R}(U+X-A)=\mathcal{R}(A)$ and $\mathcal{N}(U+X-A)=\mathcal{N}(A)$, then the unique splitting $A=B-C$ induced by $\mathcal{H}$ is a proper splitting such that $\mathcal{H}=B^{\#}C,$ where $B=A(I-\mathcal{H})^{-1}.$
\end{corollary}

The next result says that $\mathcal{H}$ and $\mathcal{S}$  induce the same splitting.

\begin{theorem}\label{thm4.9}
Let $A=K-L=U-V=X-Y$ be three proper G-weak regular splittings of type II of a group monotone matrix $A$ and $\mathcal{S}=YX^{\#}VU^{\#}LK^{\#}$. If $\mathcal{R}(K+X-A+YU^{\#}L)=\mathcal{R}(A)$ and $\mathcal{N}(K+U-A+YU^{\#}L)=\mathcal{N}(A)$, then $\mathcal{H}$ and $\mathcal{S}$ induce the same proper splitting $A=B-C$. Furthermore, the unique proper splitting $A=\overline{B}-\overline{C}$ induced by the matrix $\mathcal{S}$ is also a proper G-weak regular splitting of type II. 
\end{theorem}

\begin{proof}
By Lemma \ref{lem4.8}, we have $B=A(I-\mathcal{H})^{-1}$. Let us consider $\overline{B}=(I-\mathcal{S})^{-1}A$ and $\overline{C}=\overline{B}-A.$ Our aim is to show that the matrix $\mathcal{H}$ and $\mathcal{S}$ induce the same proper splitting $A=B-C.$
Since $\mathcal{H}=\mathcal{H}A^{\#}A$ and $\mathcal{S}=A\mathcal{H}A^{\#}$, so $\mathcal{S}^{k}=A\mathcal{H}^{k}A^{\#}$ for any  integer $k$. Since $\mathcal{S}\geq0$, by Theorem \ref{thm4.3} and Theorem \ref{thm4.4}, we have $\rho(\mathcal{S})<1.$ Again, Theorem \ref{neuman} yields
\begin{align*}
\overline{B}&=(I-\mathcal{S})^{-1}A\\
&=\sum_{k=0}^{\infty}\mathcal{S}^{k}A\\
&=\sum_{k=0}^{\infty}A\mathcal{H}^{k}A^{\#}A\\
&=\sum_{k=0}^{\infty}A\mathcal{H}^{k}\\
&=A(I-\mathcal{H})^{-1}\\
&=B.
\end{align*}
Then, $\mathcal{R}(\overline{B})=\mathcal{R}(B)=\mathcal{R}(A)$ and $\mathcal{N}(\overline{B})=\mathcal{N}(B)=\mathcal{N}(A)$. Thus, $A=\overline{B}-\overline{C}$ is a proper splitting.
Next, we show that $A=\overline{B}-\overline{C}$ is a proper G-weak regular splitting of type II. Let $Z=A^{\#}(I-\mathcal{S}).$ Then, $Z\overline{B}=A^{\#}(I-\mathcal{S})(I-\mathcal{S})^{-1}A=A^{\#}A.$ Hence $Z\overline{B}Z=A^{\#}AA^{\#}(I-\mathcal{S})=A^{\#}(I-\mathcal{S})=Z.$ Using the property $AA^{\#}\mathcal{S}=\mathcal{S}=\mathcal{S}AA^{\#},$ we obtain
\begin{align*}
    \overline{B}Z&=(I-\mathcal{S})^{-1}AA^{\#}(I-\mathcal{S})\\
      &=(I-\mathcal{S})^{-1}(AA^{\#}-AA^{\#}\mathcal{S})\\
      &=(I-\mathcal{S})^{-1}(AA^{\#}-\mathcal{S}AA^{\#})\\
      &=(I-\mathcal{S})^{-1}(I-\mathcal{S})AA^{\#}\\
      &=AA^{\#}=A^{\#}A\\
      &=Z\overline{B}.
\end{align*}
So, $\overline{B}Z=Z\overline{B}$. Now,
\begin{align*}
\overline{B}Z\overline{B}&=AA^{\#}(I-\mathcal{S})^{-1}A\\
&=AA^{\#}\sum_{k=0}^{\infty}\mathcal{S}^{k}A\\
&=\sum_{k=0}^{\infty}\mathcal{S}^{k}A=(I-\mathcal{S})^{-1}A\\
&=\overline{B}.
\end{align*}
Hence, from the proof of Theorem \ref{thm4.4}, we have $\overline{B}^{\#}=A^{\#}(I-\mathcal{S})\geq 0.$
Therefore,
$\overline{C}$ $\overline{B}^{\#}=(\overline{B}-A)\overline{B}^{\#}=\overline{B}$ $\overline{B}^{\#}-A\overline{B}^{\#}=\overline{B}$ $\overline{B}^{\#}-AA^{\#}(I-\mathcal{S})=AA^{\#}-AA^{\#}(I-\mathcal{S})=AA^{\#}-AA^{\#}+AA^{\#}\mathcal{S}=\mathcal{S}\geq 0$.
Thus, $A=\overline{B}-\overline{C}$ is a proper G-weak regular splitting of type II induced by $\mathcal{S}$. Let $A=\overline{B}_{1}-\overline{C}_{1}$ be another splitting induced by $\mathcal{S}$ such that $\mathcal{S}=\overline{C}_{1}\overline{B}_{1}^{\#}.$ Then, $\mathcal{S}\overline{B}_{1}=\overline{C}_{1}\overline{B}_{1}^{\#}\overline{B}_{1}=\overline{C}_{1}=\overline{B}_{1}-A.$ So, $A=\overline{B}_{1}-\mathcal{S}\overline{B}_{1}=(I-\mathcal{S})\overline{B}_{1}$ which further yields $\overline{B}_{1}=(I-\mathcal{S})^{-1}A=\overline{B}.$ Hence, $A=\overline{B}-\overline{C}$ is the unique proper G-weak regular splitting of type II induced by S.
 \end{proof}

The corollary stated below extends Lemma 3(ii) and Theorem 3 of \cite{climent:2003} in  the non-singular matrix case.

\begin{corollary}
Let $A=K-L=U-V=X-Y$ be three weak regular splittings of type II of a monotone matrix $A$ and $\mathcal{S}=YX^{-1}VU^{-1}LK^{-1}$. Then $\mathcal{H}$ and $\mathcal{S}$ induce the same splitting $A=B-C$. Furthermore, the unique splitting $A=\overline{B}-\overline{C}$ induced by the matrix $\mathcal{S}$ is also a weak regular splitting of type II.
\end{corollary}

Theorem \ref{thm4.9} reduces to the following when we have two proper G-weak regular splittings of type II.

\begin{corollary}
Let $A=U-V=X-Y$ be two proper G-weak regular splittings of type II of a group monotone matrix $A$ and $\mathcal{S}=YX^{\#}VU^{\#}$. If $\mathcal{R}(U+X-A)=\mathcal{R}(A)$ and $\mathcal{N}(U+X-A)=\mathcal{N}(A)$, then $\mathcal{H}$ and $\mathcal{S}$ induce the same proper splitting $A=B-C$. Furthermore, the unique proper splitting $A=\overline{B}-\overline{C}$ induced by the matrix $\mathcal{S}$ is also a proper G-weak regular splitting of type II.
\end{corollary}

In case of non-singular matrix $A$, the above result reduces to the following form, which is a part of \cite[Theorem 3]{climent:2003}.

\begin{corollary}
Let $A=U-V=X-Y$ be two weak regular splittings of type II of a  monotone matrix $A$ and $\mathcal{S}=YX^{-1}VU^{-1}$. Then $\mathcal{H}$ and $\mathcal{S}$ induce the same splitting $A=B-C$. Furthermore, the unique splitting $A=\overline{B}-\overline{C}$ induced by the matrix $\mathcal{S}$ is also a weak regular splitting of type II.
\end{corollary}

The next result shows that under some suitable assumptions, the three-step alternating scheme converges faster than the usual iteration scheme (\ref{eqn1.2}).

\begin{theorem}\label{thm4.10}
Let $A=K-L=U-V=X-Y$ be three proper G-weak regular splittings of type II of a group monotone matrix $A$ with $\mathcal{R}(K+X-A+YU^{\#}L)=\mathcal{R}(A)$ and $\mathcal{N}(K+X-A+YU^{\#}L)=\mathcal{N}(A)$. Let $A=B-C$ be the proper G-weak regular splitting of type II induced by $\mathcal{S}$. If $KB^{\#}\geq I$, $UB^{\#}\geq I$ and $XB^{\#}\geq I$, then $$\rho(\mathcal{H})\leq \min\{\rho(K^{\#}L),\rho(U^{\#}V),\rho(X^{\#}Y)\}<1.$$
\end{theorem}

\begin{proof}
Since $A=K-L=U-V=X-Y$ are three proper G-weak regular splittings of type II of a group monotone matrix $A$, so by Theorem \ref{thm3.7}, we have $\rho(K^{\#}L)<1$, $\rho(U^{\#}V)<1$ and $\rho(X^{\#}Y)<1$. Considering the splittings $A=B-C$ and $A=K-L$, and using Theorem \ref{thm2.4} $(v)$, we have
\begin{equation}\label{eqnl1}
B^{\#}(I-CB^{\#})^{-1}=K^{\#}(I-LK^{\#})^{-1}.
\end{equation}
Pre-multiplying (\ref{eqnl1}) by $K$, we obtain
\begin{equation}\label{eqnl2}
  KB^{\#}(I-CB^{\#})^{-1}=KK^{\#}(I-LK^{\#})^{-1}.  
\end{equation}
As $CB^{\#}\geq0$, there exists an eigenvector $x\geq0$ $(x\neq 0)$ such that $CB^{\#}x=\rho(CB^{\#})x$ by Theorem \ref{frob1}. Post-multiplying (\ref{eqnl2}) by $x$, and using $KK^{\#}x=x$ and $KK^{\#}L=L$, we get $KB^{\#}(I-CB^{\#})^{-1}x=KK^{\#}(I-LK^{\#})^{-1}x$, i.e., $\displaystyle \dfrac{KB^{\#}}{1-\rho(B^{\#}C)}x=KK^{\#}\sum_{i=0}^{\infty}(LK^{\#})^{i}x=KK^{\#}(I+LK^{\#}+\dots)x=KK^{\#}x+KK^{\#}LK^{\#}x+\dots=x+LK^{\#}x+\dots=(I-LK^{\#})^{-1}x$. Since $KB^{\#}\geq I$, we have
$\dfrac{x}{1-\rho(B^{\#}C)}\leq (I-LK^{\#})^{-1}x.$
Thus, $\rho(B^{\#}C)=\rho(\mathcal{H})\leq \rho(K^{\#}L)$ by Theorem \ref{frob2}. Similarly, applying the same procedure to the pair of splittings $A=B-C$ and $A=U-V$, and $A=B-C$ and $A=X-Y$, we have $\rho(\mathcal{H})\leq \rho(U^{\#}V)<1$ and $\rho(\mathcal{H})\leq \rho(X^{\#}Y)<1$, respectively. Hence $\rho(\mathcal{H})\leq \min\{\rho(K^{\#}L), ~\rho(U^{\#}V),~\rho(X^{\#}Y)\}<1$.
 \end{proof}

The next result shows that under a few assumptions, the two-step alternating iteration scheme also converges faster than the usual iteration scheme \eqref{eqn1.2}.

\begin{corollary}\label{cor3.18}
Let $A=U-V=X-Y$ be two proper G-weak regular splittings of type II of a group monotone matrix $A$ with $\mathcal{R}(U+X-A)=\mathcal{R}(A)$ and $\mathcal{N}(U+X-A)=\mathcal{N}(A)$. Let $A=B-C$ be the proper G-weak regular splitting of type II induced by the matrix $X^{\#}YU^{\#}V$. If $UB^{\#}\geq I$ and $XB^{\#}\geq I$, then $\rho(U^{\#}VX^{\#}Y)\leq \min\{\rho(U^{\#}V),\rho(X^{\#}Y)\}<1$.
\end{corollary}

The conditions $KB^{\#}\geq I$, $UB^{\#}\geq I$ and $XB^{\#}\geq I$ can be dropped if the given splittings are proper G-weak regular splittings of both types. This is shown below.

\begin{theorem}\label{thm3.13}
Let $A=K-L=U-V=X-Y$ be three proper G-weak regular splittings of both types of a group monotone matrix $A$ with $\mathcal{R}(K+X-A+YU^{\#}L)=\mathcal{R}(A)$ and $\mathcal{N}(K+X-A+YU^{\#}L)=\mathcal{N}(A)$. Then $$\rho(\mathcal{H})\leq \min\{\rho(K^{\#}L),\rho(U^{\#}V),\rho(X^{\#}Y)\}<1.$$
\end{theorem}
\begin{proof}
Let $A=B-C$ be a splitting induced by $S$. Since $A=K-L=U-V=X-Y$ are proper G-weak splittings of type II, so $A=B-C$ is proper G-weak regular splitting of type II by Lemma \ref{lem4.8}. From \eqref{eqn4.1} and \eqref{eqn4.2}, we have
\begin{align*}
   B^{\#}=X^{\#}+X^{\#}YK^{\#}+X^{\#}YK^{\#}LU^{\#}\geq X^{\#}
\end{align*}
and 
\begin{align*}
B^{\#}=K^{\#}+X^{\#}LK^{\#}+X^{\#}YU^{\#}LK^{\#}\geq K^{\#}.
\end{align*}
Applying Theorem \ref{thm3.8} to the pair of splittings $A=B-C$ and $A=X-Y$, and $A=B-C$ and $A=K-L$, we get $\rho(\mathcal{H})\leq \rho(X^{\#}Y)$ and $\rho(\mathcal{H})\leq \rho(K^{\#}L)$, respectively. Again,
\begin{align*}
   B^{\#}&=X^{\#}(YU^{\#}VK^{\#}+YU^{\#}+I)\\
         &=X^{\#}YU^{\#}VK^{\#}+X^{\#}XU^{\#}-X^{\#}AU^{\#}+X^{\#}UU^{\#}\\
         &=X^{\#}YU^{\#}VK^{\#}+U^{\#}+X^{\#}(U-A)U^{\#}\\
         &=X^{\#}YU^{\#}VK^{\#}+U^{\#}+X^{\#}VU^{\#}\\
         &\geq U^{\#}.
\end{align*}
So, by Theorem \ref{thm3.8}, we have $\rho(\mathcal{H})=\rho(B^{\#}C)\leq \rho(U^{\#}V)$. Hence $$\rho(\mathcal{H})\leq \min\{\rho(K^{\#}L), \rho(U^{\#}V),~\rho(X^{\#}Y)\}<1.$$
 \end{proof}

 The group inverse version of \cite[Theorem 4.13]{giri:2017}, which applies to the two-step alternating iteration scheme \eqref{eqn1.7}, is obtained as a corollary below.

\begin{corollary}\label{cor3.21}
Let $A=U-V=X-Y$ be two proper G-weak regular splittings of both types of a group monotone matrix $A$ with $\mathcal{R}(U+X-A)=\mathcal{R}(A)$ and $\mathcal{N}(U+X-A)=\mathcal{N}(A)$. Then $$\rho(U^{\#}VK^{\#}L)\leq \min\{\rho(U^{\#}V),\rho(X^{\#}Y)\}<1.$$
\end{corollary}

We next show that the iteration matrix corresponding to the three-step alternating iteration scheme converges faster than the iteration matrix corresponding to the two-step alternating iteration scheme in case of proper G-weak regular splittings of type II under a few assumptions.

\begin{theorem}\label{thm4.11}
Let $A=K-L=U-V=X-Y$ be three proper G-weak regular splittings of type II of a group monotone matrix $A$ with $\mathcal{R}(K+X-A+YU^{\#}L)=\mathcal{R}(A)$ and $\mathcal{N}(K+X-A+YU^{\#}L)=\mathcal{N}(A)$. Let $A=B-C$, $A=B_{12}-C_{12}$, $A=B_{13}-C_{13}$ and $A=B_{23}-C_{23}$ be proper G-weak regular splittings of type II induced by the matrices $\mathcal{S}$, $U^{\#}VK^{\#}L$, $X^{\#}YK^{\#}L$, and $X^{\#}YU^{\#}V$, respectively. If $B_{12}B^{\#}\geq I$, $B_{13}B^{\#}\geq I$ and $B_{23}B^{\#}\geq I$, then $$\rho(\mathcal{H})\leq
\min\{\rho(K^{\#}LU^{\#}V), \rho(X^{\#}YK^{\#}L ), \rho(X^{\#}YU^{\#}V)\} < 1.$$
\end{theorem}

\begin{proof}
Using the same argument as in Theorem \ref{thm4.10} to the pair of splittings $A=B-C$ and $A=B_{12}-C_{12}$, $A=B-C$ and $A=B_{13}-C_{13}$, and $A=B-C$ and $A=B_{23}-C_{23}$, we get
$$\rho(\mathcal{H})\leq
\min\{\rho(K^{\#}LU^{\#}V), \rho(X^{\#}YK^{\#}L ), \rho(X^{\#}YU^{\#}V)\} < 1.$$
 \end{proof}

The following is the non-singular version of Theorem \ref{thm4.11}.

\begin{corollary}\label{cor3.23}
Let $A=K-L=U-V=X-Y$ be three weak regular splittings of type II of a  monotone matrix $A$. Let $A=B-C$, $A=B_{12}-C_{12}$, $A=B_{13}-C_{13}$ and $A=B_{23}-C_{23}$ be weak regular splittings of type II induced by the matrices $\mathcal{S}$, $U^{-1}VK^{-1}L$, $X^{-1}YK^{-1}L$ and $X^{-1}YU^{-1}V$, respectively. If $B_{12}B^{-1}\geq I$, $B_{13}B^{-1}\geq I$ and $B_{23}B^{-1}\geq I$, then $$\rho(\mathcal{H})\leq
\min\{\rho(K^{-1}LU^{-1}V), \rho(X^{-1}YK^{-1}L), \rho(X^{-1}YU^{-1}V)\} < 1.$$
\end{corollary}

\begin{example}
Consider the two-dimensional Laplace's equation
\begin{align}\label{GPE}
\frac{\partial^2 u}{\partial x^2}
      + \frac{\partial^2 u}{\partial y^2} = 0, ~~~~\hspace{0.4cm}0\leq x\leq 1,~0\leq y\leq 1
\end{align}
with boundary conditions
$$u(x, y)|_{\partial R} = x+y+xy.$$
Let the square region $R=\left\{(x,y):0\leq x\leq 1,~0\leq y\leq 1\right\}$ be covered by a grid with sides
parallel to the coordinate axis and with an equal grid spacing $h=\Delta x=\Delta y$. If
$Nh = 1$, then the number of internal grid points is $(N - 1)^2$. The finite difference method using the $\mathcal{O}(h^2)$ central difference discretization on uniform grids generates the linear system $Ax = b$, where $A$ is of order $(N-1)^2\times (N-1)^2$ and $b$ is the right-hand side vector derived from the Dirichlet boundary conditions. 
The coefficient matrix $A$ is of the form $A=I\kronecker J +J\kronecker I.$
Here, $\kronecker$ is the  Kronecker product, $I$ is an identity matrix of order $(N-1)\times (N-1)$ and 
$J=tridiagonal\left(-1,4,-1\right)$ is of order $(N-1)\times (N-1)$.\\
Setting $K=diag(A)$, $U=1.5diag(A)$ and $X=1.75diag(A)$, we get three weak regular splittings of type II $A=K-L=U-V=X-Y$  of the monotone matrix $A$. Using these splittings, we have compared the computational aspects of three-step, two-step, and single-step iterative schemes in Table \ref{tab:1}.
\end{example}

\begin{table}[H]
    \centering
     \caption{Comparison table}
    \begin{tabular}{ccccccc}
    \hline
     Order &Splitting &IT &$||b-Ax_n||_2$ &$||A^{-1}b-x_n||_2$ &Time &$\rho$\\
        \hline
         \multirow{3}{*} {400} &\tiny{\bf Three-step} &672
 &4.4511e-08 &9.9629e-07  &0.02  &0.9752\\
       &\tiny{\bf Two-step} &902 &4.4374e-08 &9.9323e-07  &0.03  &0.9815 \\
       &\tiny{\bf Single-step} &1502
 &4.4629e-08  &9.9893e-07  &0.12  &0.9888\\
        \hline
        \multirow{3}{*} {1600} &\tiny{\bf Three-step} & 2669 &1.1661e-08 &9.9350e-07  &5.83  &0.9934\\
       &\tiny{\bf Two-step} &3583 &1.1688e-08 &9.9583e-07  &8.19  &0.9951 \\
       &\tiny{\bf Single-step} &5970 &1.1717e-08
 &9.9831e-07  &16.28  &0.9971\\
        \hline
        \multirow{3}{*} {4900} &\tiny{\bf Three-step} &8256  &3.9092e-09 &9.9849e-07  &177.05  &0.9978\\
       &\tiny{\bf Two-step} &11086
 &3.9092e-09 &9.9848e-07  &244.08  &0.9984 \\
       &\tiny{\bf Single-step} &18475
 &3.9123e-09  &9.9929e-07  &515.14  &0.9990\\
        \hline
        \multirow{3}{*} {6400} &\tiny{\bf Three-step} & 10824 &3.0051e-09 &9.9896e-07  &386.60  &0.9983\\
       &\tiny{\bf Two-step} &14534
 &3.0067e-09 &9.9952e-07  &567.94  &0.9987 \\
       &\tiny{\bf Single-step} &24222 &3.0078e-09
 &9.9989e-07  &1041.26  &0.9992\\
        \hline
       \multirow{3}{*} {10000} &\tiny{\bf Three-step} &17034 &1.9329e-09 &9.9896e-07  &1509.86  &0.9989\\
       &\tiny{\bf Two-step} &22873 &1.9338e-09 &9.9943e-07  &2261.70  &0.9992 \\
       &\tiny{\bf Single-step} &38120 &1.9345e-09 &9.9981e-07  &4043.28  &0.9995\\
        \hline
 \multirow{3}{*} {12100} &\tiny{\bf Three-step} &20679 &1.6014e-09 &9.9966e-07  &3041.72  &0.9991\\
       &\tiny{\bf Two-step} &27768 &1.6017e-09 &9.9986e-07  &3674.84  &0.9993 \\
       &\tiny{\bf Single-step} &46279 &1.6018e-09 &9.9991e-07  &6965.13  &0.9996\\
        \hline
    \end{tabular}
    \label{tab:1}
\end{table}

\section{Semiconvergence of three-step alternating scheme}\label{sec4}
This section is divided into two subsections. In the subsection \ref{subsec4.1}, we discuss the semiconvergence of three-step alternating iterative scheme when the coefficient matrix $A$ is a singular $M$-matrix and the splittings are regular (i.e., $U^{-1}\geq 0$ and $V\geq 0$) or weak regular type I (i.e., $U^{-1}\geq 0$ and $U^{-1}V\geq 0$) or weak regular type II (i.e., $U^{-1}\geq 0$ and $VU^{-1}\geq 0$). In the subsection \ref{subsec4.2}, the coefficient matrix $A$ is only singular and the splittings are quasi-regular or quasi-weak regular type I (or type II).

\subsection{When $A$ is a singular $M$-matrix}\label{subsec4.1}

The following result provides a characterization of a semiconvergent matrix.

\begin{theorem}\label{semi}\textnormal{(\cite{neumann})}\\
Let $T\in \mathbb{R}^{n\times n}$ with $\rho(T)=1$. Then, $T$ is semiconvergent if and only if the following statements hold:
\begin{itemize}
    \item[(i)] $1\in \sigma(T)$ and $\gamma(T)<1$;
    \item[(ii)] $\mathcal{N}(I-T)\oplus \mathcal{R}(I-T)=\mathbb{R}^n$.
\end{itemize}
 \end{theorem}
Recall that the last condition $(ii)$ is equivalent to the existence of the group inverse of the matrix $(I-T)$.

Alefeld \cite{Alefeld} obtained the following result for semiconvergence of a nonnegative matrix with positive diagonal entries.
\begin{theorem}\label{alefeld}\textnormal{(\cite[Theorem 2]{Alefeld})}\\
Let $T\geq 0$, $diag(T)>0$. Then, $T$ is semiconvergent if $\rho(T)\leq1$ and $\mathcal{N}(I-T)\oplus \mathcal{R}(I-T)=\mathbb{R}^n$.
\end{theorem}

The next result gives a characterization of an $M$-matrix with property $c$ whenever $A$ has a regular splitting.

\begin{theorem}\label{thm5.1}\textnormal{(\cite{neumann} \& \cite[Theorem 10]{Migall})}\\
Let $A=U-V$ be a regular splitting. Then, $A$ is an M-matrix with property c if and only if $\rho(T)\leq1$ and $\mathcal{N}(I-T)\oplus \mathcal{R}(I-T)=\mathbb{R}^n$, where $T=U^{-1}V$.
\end{theorem}

Consider three regular splittings of the singular matrix $A$, i.e., $A=K-L=U-V=X-Y$, then re-writing \eqref{eqn1.6}, we have
\begin{equation}\label{eqnsemi}
    x^{k+1}=X^{-1}YU^{-1}VK^{-1}Lx^{k}+X^{-1}(YU^{-1}VK^{-1}+YU^{-1}+I)b,\quad k=0,1,2,\ldots,
\end{equation}

If $\mathcal{H}=X^{-1}YU^{-1}VK^{-1}L$, then $I-\mathcal{H}=K^{-1}(K+X-A+YU^{-1}L)X^{-1}A$. Thus, in order to satisfy $\mathcal{N}(I-\mathcal{H})=\mathcal{N}(A)$, we must have $(K+X-A+YU^{-1}L)$ as nonsingular. The following theorem is a direct extension of \cite[Theorem 3.4]{Benzi} and will be crucial for proving the main results in this section.
\begin{theorem}\label{thm5.2}
Let $A=K-L=U-V=X-Y$ be three regular splittings of a singular matrix $A$. If ($K+X-A+YU^{-1}L$) is nonsingular, then there exists a regular splitting $A=B-C$ such that $B^{-1}C=X^{-1}YU^{-1}VK^{-1}L$.
\end{theorem}

Now, we provide sufficient conditions for the semiconvergence of \eqref{eqnsemi}.

\begin{theorem}\label{t4.5}
Let A be an M-matrix with property c and $A=K-L=U-V=X-Y$ be three regular splittings of $A$. If ($K+X-A+YU^{-1}L$) is nonsingular and $diag(X^{-1}YU^{-1}VK^{-1}L)>0$, then the matrix $X^{-1}YU^{-1}VK^{-1}L$ is semiconvergent.
\end{theorem}
\begin{proof}
By Theorem \ref{thm5.2}, there exists a regular splitting $A=B-C$ such that $B^{-1}C=X^{-1}YU^{-1}VK^{-1}L$. Applying Theorem \ref{thm5.1}, we get $\rho(B^{-1}C)\leq 1$ and $\mathcal{N}(I-B^{-1}C)\oplus \mathcal{R}(I-B^{-1}C)=\mathbb{R}^n$. Since $B^{-1}C\geq 0$ and $diag(B^{-1}C)>0$, by Theorem \ref{alefeld}, $B^{-1}C=K^{-1}LU^{-1}VX^{-1}Y$ is semiconvergent.
 \end{proof}

\begin{remark}
The assumption $diag(K^{-1}LU^{-1}VX^{-1}Y)>0$ need not hold in certain matrix splittings, then in order to make sure that 1 is the only eigenvalue on the unit circle, we seek the help of a standard method called shifting of matrix. Our next results are in this direction.
\end{remark}

\begin{theorem}\label{thm5.6}
Let A be an M-matrix with property c and $A=K-L=U-V=X-Y$ be three regular splittings of $A$. Suppose that ($K+X-A+YU^{-1}L$) is nonsingular. Then, for each $\delta\in(0,1)$, the matrix $\mathcal{H}_{\delta}=\delta \mathcal{H}+(1-\delta)I$ is semiconvergent.
\end{theorem}
\begin{proof}
Clearly, by Theorem \ref{thm5.1} and Theorem \ref{thm5.2}, we have $\rho(\mathcal{H}_{\delta})\leq 1$ and $\mathcal{N}(I-\mathcal{H}_{\delta})\oplus \mathcal{R}(I-\mathcal{H}_{\delta})=\mathbb{R}^n$, where $\delta \in (0,1)$. Therefore, $\mathcal{H}_\delta$ has only the eigenvalue 1 on the unit circle. Now, applying Theorem \ref{semi}, we conclude that $\mathcal{H}_{\delta}$ is semiconvergent for all $\delta\in(0,1)$.
 \end{proof}

If required, we can replace (see \cite{Migall}) \eqref{eqnsemi} by
\begin{equation}\label{eqnsemi2}
    x^{k+1}=\delta(\mathcal{H}x^k+\hat{b})+(1-\delta)x^k,\quad k=0,1,2,\ldots,
\end{equation}
where $\mathcal{H}=X^{-1}YU^{-1}VK^{-1}L$ and $\hat{b}=X^{-1}(YU^{-1}VK^{-1}+YU^{-1}+I)b$.
\begin{theorem}\label{t4.8}
Let A be an M-matrix with property c and $A=K-L=U-V=X-Y$ be three regular splittings of $A$ such that ($K+X-A+YU^{-1}L$) is nonsingular. Then, the following conditions hold:
\begin{itemize}
    \item[(i)] if $diag(\mathcal{H})>0$, then the three-step alternating scheme \eqref{eqnsemi} converges
to a solution of the consistent linear system $Ax=b$, for any initial
vector $x^0$,

\item [(ii)] the modified three-step alternating scheme \eqref{eqnsemi2} converges
to a solution of the consistent linear system $Ax=b$, for any initial
vector $x^0$.
\end{itemize}
\end{theorem}
\begin{proof}
(i) Let $x^*$ be a solution of $Ax=b$ and $e_{k+1}$ be the error vector at $(k+1)^{th}$ iterations, then $e_{k+1}=\mathcal{H}e_k=\mathcal{H}^ke_0$. By Theorem \ref{t4.5}, $\mathcal{H}$ is semiconvergent and thus, $\displaystyle \lim_{k\to 0}e_{k+1}=\lim_{k\to 0}\mathcal{H}^ke_0=(I-(I-\mathcal{H})(I-\mathcal{H})^{\#})e^0.$\par
(ii) This is trivial from Theorem \ref{thm5.6}.
 \end{proof}

\begin{example}\label{ex6.15}\textnormal{(\cite[Example 8.2.3]{bpbook})}\\
Consider an $n$-state random
walk of a particle that moves in a straight line with a unit step. Suppose the particle moves to the left or right one step
with equal probability, except at the boundary points where the particle
always moves back one step with probability one. If it moves to the right with probability $p$ and to the left with probability $q$, then here $p = q$
The transition
matrix associated with this Markov chain is the $n\times n$ matrix $\mathcal{T}$ given by
$$\mathcal{T}=\begin{bmatrix}
0  &1   & &  & &0\\
1/2  &0  &1/2 & &\\
      &\ddots &\ddots &\ddots &\\ 
      &   &\ddots &\ddots &\ddots  &\\
      &   & &1/2 &0 &1/2\\
0     &     & & &1 &0  
\end{bmatrix}.$$
\noindent The stationary probability vector $x$ (say) satisfy $x=\mathcal{T}^tx$ or equivalently, we need to solve the homogeneous system $(I-\mathcal{T}^t)x=0$. Clearly, $I-\mathcal{T}^t$ is a singular $M$-matrix with property $c$ and column-sum zero. For computing $x$, we apply above discussed semiconvergence theory for a three-step alternating iterative method and compare it with single-step and two-step. For this purpose, we vary the order $n$ of the coefficient matrix $I-\mathcal{T}^t$. 
Setting $K=2diag(A)$, $U=2.5diag(A)$ and $X=3diag(A)$, we get three regular splittings $A=K-L=U-V=X-Y$ of the singular $M$-matrix $A$ with property $c$. Table \ref{table:2} shows that the computational time and number of iterations taken by the three-step alternating iterative method is quite less as compared to the single-step and two-step alternating methods. Here, the tolerance is $\epsilon=10e-8$.
\end{example}
\begin{table}[H]
    \centering
     \caption{Comparison table}
    \begin{tabular}{cccccc}
    \hline
     Size &Splitting &IT &Time &$\gamma$\\
        \hline
         \multirow{3}{*} {$10\times 10$} &{\bf Three-step} &166
  &0.000421 &0.9274  \\
       &{\bf Two-step} &228   &0.000867 &0.9465  \\
       &{\bf Single-step} &409 &0.001413  &0.9698\\
        \hline
        \multirow{3}{*} {$30\times 30$} &{\bf Three-step} &1330 &0.008085 & 0.9928  \\
       &{\bf Two-step} &1822   &0.010165 &0.9947  \\
       &{\bf Single-step} &3279   &0.013322  &0.9971\\
        \hline
        \multirow{3}{*} {$90\times 90$} &{\bf Three-step} &8894  &0.086668  &0.9992\\
      &{\bf Two-step} &12188  &0.106848  &0.9994 \\
      &{\bf Single-step} &21937
  &0.186011  &0.9997\\
        \hline
    \end{tabular}\label{table:2}
\end{table}

\subsection{When $A$ is a singular matrix}\label{subsec4.2}
Wu \cite{wu:2013} obtained the semiconvergence of the two-step alternating iterative scheme in the case of quasi-weak regular splittings of type I (or type II). This subsection establishes semiconvergence of the three-step alternating iterative scheme when the coefficient matrix $A$ is singular, and the splittings are quasi-regular and quasi-weak regular of type I (or type II) splittings. But first, we recall the definitions of these classes of splittings.  
\begin{definition}\textnormal{(\cite[Definition 2.3]{wu:2013})}\\
 Let $A=U-V$ be a splitting of $A\in \mathbb{R}^{n\times n}$. Assume that $index(I-U^{-1}V)\leq 1$ and $index(I-VU^{-1})\leq 1$. Let $K_1=(I-U^{-1}V)(I-U^{-1}V)^{\#}$ and $K_2=(I-VU^{-1})^{\#}(I-VU^{-1})$. Then, $A=U-V$ is called 
     \begin{enumerate}
         \item {\it a quasi regular splitting} if $U^{-1}$ exists, $U^{-1}\geq 0$ and $VK_1\geq 0$,
         \item {\it a quasi weak regular splitting of type I}  if  $U^{-1}$ exists, $U^{-1}\geq 0$ and $U^{-1}VK_1\geq 0$,
         \item {\it a quasi weak regular splitting of type II}  if $U^{-1}$ exists, $U^{-1}\geq 0$ and $K_2VU^{-1}\geq 0$.
     \end{enumerate}
     \end{definition} 
     
Wu \cite{wu:2013} obtained the following result for semiconvergence of a two-step alternating iterative scheme without the condition `$A$ is an $M$-matrix' in the case of quasi-weak regular splittings of type I (or type II).
\begin{theorem}\label{thm6.9}\textnormal{(\cite[Theorem 3.1]{wu:2013})}\\
Let $A = K-L= U-V$
be two quasi-weak regular splittings of type I (or type II) of a singular matrix $A$ such that any pair of splittings $A = K-L= U-V$ is semiconvergent with $index(K^{-1}L)\leq 1$, $index(U^{-1}V)\leq 1$ and $index(U^{-1}VK^{-1}L)\leq 1$. Then, $\mathcal{H}= U^{-1}VK^{-1}L$ is semiconvergent. Furthermore, the splitting $A = B-C$ induced by $\mathcal{H}$ is  quasi-weak regular of the same type.
\end{theorem}

Similar to the above result, we can obtain the semiconvergence of a three-step alternating iterative scheme \eqref{eqnsemi} when we have three quasi-weak regular splittings of type I (or type II). The same result also holds if we replace quasi-weak regular splittings of type I (or type II) by quasi-regular splitting.
\begin{theorem}\label{t4.10}
Let $A = K-L= U-V=X-Y$
be three quasi-weak regular splittings of type I (or type II) of a singular matrix $A$ such that any pair of splittings $A = K-L= U-V$ is semiconvergent with $index(K^{-1}L)\leq 1$, $index(U^{-1}V)\leq 1$, $index(X^{-1}Y)\leq 1$, $index(I-U^{-1}VK^{-1}L) \leq 1$ and $index(X^{-1}YU^{-1}VK^{-1}L)\leq 1$. Then, $\mathcal{H}= X^{-1}YU^{-1}VK^{-1}L$ is semiconvergent. Furthermore, the splitting $A = B-C$ induced by $\mathcal{H}$ is a quasi-weak regular splitting of the same type.
\end{theorem}
\begin{proof}
We will prove this result for quasi-weak regular splittings of type I. The type II case similarly follows. Since, any one of the splittings $A = K-L= U-V$ is semiconvergent quasi weak regular splitting of type I and, $index(K^{-1}L)\leq 1$, $index(U^{-1}V)\leq 1$ and $index(I-U^{-1}VK^{-1}L) \leq 1$, we have  $T=U^{-1}VK^{-1}L$ is semiconvergent by Theorem \ref{thm6.9}, and the splitting $A=B_{12}-C_{12}$ (say) induced by $T$ is also a semiconvergent quasi weak regular splitting of type I. Now, if we consider the splittings $A=B_{12}-C_{12}=X-Y$ and apply Theorem \ref{thm6.9}, then the matrix $\mathcal{H}=X^{-1}YB_{12}^{-1}C_{12}= X^{-1}YU^{-1}VK^{-1}L$ is semiconvergent. Further, by the same argument as given earlier, the induced splitting $A=B-C$ by $\mathcal{H}$ is also quasi-weak regular splitting of type I.
 \end{proof}

Wu \cite{wu:2013} then obtained the following comparison result.

\begin{theorem}\label{thm6.11}\textnormal{(\cite[Theorem 4.1]{wu:2013})}\\
Let $A = K-L$
be a semiconvergent quasi-regular splitting and $A =U-V$
be a quasi regular splittings of a singular matrix $A$ such that $index(I-K^{-1}L)\leq 1$, $index(I-U^{-1}V) \leq 1$ and $index(I-\mathcal{H}) \leq 1$, where $\mathcal{H}= U^{-1}VK^{-1}L$. Then,
$$\gamma(\mathcal{H}) \leq min \{\gamma(U^{-1}V),\gamma(K^{-1}L))\} < 1.$$ 
\end{theorem}

The next result is motivated by the above theorem.

\begin{theorem}\label{thm6.12}
Let $A = K-L$
be a semiconvergent quasi-regular splitting and $A = U-V=X-Y$
be a quasi weak regular splitting of type I of a singular matrix $A$ such that $index(I-K^{-1}L)\leq 1$, $index(I-U^{-1}V) \leq 1$, $index(I-X^{-1}Y)\leq 1$ and $index(I-\mathcal{H}) \leq 1$, where $\mathcal{H}= X^{-1}YU^{-1}VK^{-1}L$. Then,
$\gamma(\mathcal{H}) \leq \gamma(X^{-1}Y) < 1.$ 
\end{theorem}
\begin{proof}
Since $A=K-L=U-V$ are quasi regular splittings, therefore the splitting $A=B_{12}-C_{12}$ induced by $U^{-1}VK^{-1}L$ is also a quasi regular splitting. Now, consider the regular splittings $A=B_{12}-C_{12}=X-Y$, then the induced splitting $A=B-C$ of $X^{-1}YB_{12}^{-1}C_{12}=X^{-1}YU^{-1}VK^{-1}L$ is also a regular splitting. Now, by Theorem \ref{thm6.11}, we have $\gamma(\mathcal{H})=\gamma(B^{-1}C)\leq \gamma(X^{-1}Y)<1.$  
 \end{proof}

Now, we obtain the next result as a corollary.

\begin{corollary}\label{cor6.13}
Let $A = K-L= U-V=X-Y$
be three semiconvergent quasi regular splittings of a singular matrix $A$ such that $index(I-K^{-1}L)\leq 1$, $index(I-U^{-1}V) \leq 1$, $index(I-X^{-1}Y)\leq 1$ and $index(I-\mathcal{H}) \leq 1$, where $\mathcal{H}= X^{-1}YU^{-1}VK^{-1}L$. Then,
$\gamma(\mathcal{H}) \leq min \{\gamma(U^{-1}V),\gamma(X^{-1}Y),\gamma(K^{-1}L))\} < 1.$ 
\end{corollary}
\begin{proof}
By Theorem \ref{thm6.12}, we have $\gamma(H) \leq \gamma(X^{-1}Y)$. Similarly, considering the other pairs of splittings and repeating the step as in Theorem \ref{thm6.12}, we conclude that $\gamma(\mathcal{H}) \leq min \{\gamma(U^{-1}V),\gamma(X^{-1}Y),\gamma(K^{-1}L)\} < 1.$
 \end{proof}

The following result provides some sufficient conditions such that the semiconvergence of the three-step alternating iterative scheme is faster than the two-step alternating iterative scheme.

\begin{theorem}\label{t4.14}
Let $A = K-L = U-V=X-Y$
be three semiconvergent quasi-regular splittings of a singular matrix $A$. Let $A=B-C=B_{12}-C_{12}=B_{13}-C_{13}=B_{23}-C_{23}$ be the quasi regular splittings induced by the matrices $\mathcal{H}=X^{-1}YU^{-1}VK^{-1}L$, $U^{-1}VK^{-1}L$, $X^{-1}YK^{-1}L$ and $X^{-1}YU^{-1}V$, respectively. Suppose that $index(I-K^{-1}L)\leq 1$, $index(I-U^{-1}V) \leq 1$, $index(I-X^{-1}Y)\leq 1$, $index(I-B_{12}^{-1}C_{12})\leq 1$, $index(I-B_{13}^{-1}C_{13})\leq 1$, $index(I-B_{23}^{-1}C_{23})\leq 1$ and $index(I-\mathcal{H}) \leq 1$. Then,\\
$$\gamma(\mathcal{H}) \leq min \{\gamma(U^{-1}VK^{-1}L),\gamma(X^{-1}YU^{-1}V),\gamma(X^{-1}YK^{-1}L))\} < 1.$$ 
\end{theorem}
\begin{proof}
Considering the pair of splittings $A=B_{12}-C_{12}=X-Y$ and applying Corollary \ref{cor6.13}, we get $\gamma(\mathcal{H})\leq \gamma(B_{12}^{-1}C_{12})=\gamma(U^{-1}VK^{-1}L)$. Similarly, taking the pairs $A=B_{13}-C_{13}=U-V$ and $A=B_{23}-C_{23}=K-L$, we get $\gamma(\mathcal{H})\leq \gamma(B_{13}^{-1}C_{13})=\gamma(X^{-1}YK^{-1}L)$ and $\gamma(\mathcal{H})\leq \gamma(B_{23}^{-1}C_{23})=\gamma(X^{-1}YU^{-1}V)$, respectively. Hence, $\gamma(\mathcal{H}) \leq min \{\gamma(U^{-1}VK^{-1}L),\gamma(X^{-1}YU^{-1}V),\gamma(X^{-1}YK^{-1}L))\} < 1.$ 
 \end{proof}

\section{Conclusion}
\noindent We have established the convergence theory of the three-step alternating iterative scheme for proper G-weak regular splittings of type II  to find an iterative solution of an index 1 linear system. In particular, we have obtained sufficient conditions for the convergence of the three-step alternating iteration scheme in Theorem \ref{thm4.4}. The uniqueness of a proper splitting induced by $S$  is shown next in Theorem \ref{thm4.9}. Finally, we have proved that the three-step alternating iteration scheme converges faster than the usual iteration scheme and the two-step alternating iteration scheme in Theorem \ref{thm4.11} under some assumptions. The computational benefits of using the three-step alternating iterative scheme are demonstrated in Table \ref{tab:1}. Next, we have established the semiconvergence of the three-step alternating iterative scheme when we have regular and weak regular splittings of type I (Theorems \ref{t4.5} and \ref{t4.8}) and quasi-regular splittings (Theorem \ref{t4.10}). Further, it is shown that the rate of semiconvergence of the three-step alternating iterative scheme is more than the two-step and single-step iterative schemes (Theorems \ref{t4.14} and \ref{thm6.12}). This theory is well suited for solving a linear system that arrives from the Markov process. Table \ref{table:2} shows that the three-step iterative schemes are effective concerning computational storage, computational time, and iteration steps. The motivation behind the procedure is to acquaint researchers and practitioners with an alternative approach that is easy to understand, implement, and computationally efficient.



\end{document}